\documentclass[12pt, reqno]{amsart}
\usepackage{mathrsfs}
\usepackage{amsfonts}
\usepackage{}
\usepackage{amsmath}
\usepackage{amsthm}
\usepackage{amssymb}
\usepackage[compress,numbers]{natbib}
\usepackage{graphicx}
\usepackage{comment}
\usepackage{tikz}
\newtheorem{theorem}{Theorem}[section]
\newtheorem{lemma}[theorem]{Lemma}

\theoremstyle{definition}
\newtheorem{definition}[theorem]{Definition}

\theoremstyle{remark}
\newtheorem{remark}[theorem]{Remark}

\numberwithin{equation}{section}

\begin{document}
	
	\title{The $L_{p}$ Gaussian  Minkowski  problem for $C$-pseudo-cones }
	
	\author{Junjie Shan$^{*}$, Wenchuan Hu }
	
	\address{School of Mathematics, Sichuan University, Chengdu, Sichuan, 610064, P. R. China}

	\email{shanjjmath@163.com,   wenchuan@scu.edu.cn
	}

	\thanks{{\it 2010 Math Subject Classifications}:  52A40, 52A38}
	

	\thanks{ $^{*}$ Corresponding author}
	\begin{abstract}
	
	The $L_{p}$  Gaussian Minkowski problem for $C$-pseudo-cones is studied in this paper, and the  existence and uniqueness  results are established. This extends our previous work on the Minkowski problem  for $C$-pseudo-cones with respect to the Gaussian surface area measure ($p=1$) and the Gaussian cone measure ($p=0$).
	\end{abstract}
	
	\maketitle

	\section{Introduction}
For convex bodies,	the Minkowski problem  aims to establish necessary and sufficient conditions for a Borel measure on the unit sphere to be the surface area measure.  The $L_p$ surface area measure is a generalization of the classical surface area measure. The $L_p$ Minkowski problem, which characterizes the $L_p$ surface area measure, was first introduced by Lutwak \cite{Lutwak1993,Lutwak1996}, extending the classical Brunn-Minkowski theory into a broader framework. When $p = 1$, the $L_p$ Minkowski problem reduces to the classical Minkowski problem. Consequently, many classical results in convex geometry become part of the more general $L_p$ theory, such as the affine isoperimetric inequality \cite{CGlpine,HFGlpine,lyz2002}, the Minkowski problem \cite{LYZ2013JAMS,HLYZ2018,LLL2022,lyz2004,lyz2018,zhu2015Lp}. Notably, for $p \geq 1$, many forms of the $L_p$ Minkowski problem have been resolved. However, significant open questions remain for the case $p < 1$.
	
	A special class of unbounded closed convex sets, known as $C$-pseudo-cones, has recently attracted considerable research interest. Within the copolarity framework, $C$-pseudo-cones are the counterparts of convex bodies with the origin in their interiors \cite{ASWdual,Rcopolar,schneider pescone,xulileng 2023}. Let $C \subset \mathbb{R}^n$ be a pointed (i.e., containing no lines), $n$-dimensional closed convex cone. A \textit{pseudo-cone} is defined as a nonempty closed convex set $K$ that does not contain the origin and satisfies $\lambda K \subseteq K$ for all $\lambda \geq 1$. The \textit{recession cone} of $K$, denoted by $\operatorname{rec} K$, is given by
	\[
	\operatorname{rec} K = \{ z \in \mathbb{R}^n : K + z \subseteq K \}.
	\]
	A \textit{$C$-pseudo-cone} is a pseudo-cone $K$ whose recession cone satisfies $\operatorname{rec} K = C$. As $x \in \mathbb{R}^n$ tends to infinity, the shape of the $C$-pseudo-cone $K$ is strongly governed by the cone $C$.
	
	The \textit{polar cone} of $C$, denoted by $C^\circ$, is defined as
	\[
	C^\circ = \{ x \in \mathbb{R}^n : \langle x, y \rangle \leq 0 \text{ for all } y \in C \},
	\]
	we will write
	$$
	 \Omega_{C^\circ} = S^{n-1} \cap \operatorname{int} C^\circ.
	$$
	
	Unlike convex bodies, the surface area measure of a $C$-pseudo-cone may be infinite. The Minkowski problem for the surface area measure and cone volume measure of $C$-pseudo-cones was first studied by Schneider \cite{schneider 2018,schneider2021}. To generate finite measures, Schneider \cite{schneider weighted,schneider weighted cone} introduced the weighted Minkowski problem, where the weighted functions are homogeneous. Furthermore, Yang, Ye, and Zhu \cite{yangyezhuLp2022} later explored the $L_p$-Brunn-Minkowski theory for $C$-coconvex sets, while Li, Ye, and Zhu \cite{Lyzdualcone} developed dual versions. Further developments in the theory of pseudo-cones can be found in \cite{ai lpdual,sz2024asyp,Wang2024asymptotic,zhang asyp}.
	
	In \cite{shan2025}, the first  author, the second  author, and Xu introduced a class of non-homogeneous measures---the Gaussian surface area measure and the Gaussian cone measure for $C$-pseudo-cones, and posed the corresponding  Gaussian Minkowski problem.  The Gaussian surface area measure of  a  $C$-pseudo-cone $K$ in $\mathbb{R}^{n}$  is defined on $\Omega_{C^{\circ}}$ as
	\begin{equation}\label{Gmeasure}
		S_{\gamma^{n}}(K,\eta)=\frac{1}{(\sqrt{2 \pi})^{n}} \int_{\nu_{K}^{-1}(\eta)} e^{-\frac{|x|^{2}}{2}} d \mathcal{H}^{n-1}(x)
	\end{equation}
	for every Borel set $\eta\subset\Omega_{C^{\circ}}$, where  $\nu_{K}$ is the outer unit normal vector  of $K$. The measure $	S_{\gamma^{n}}(K,\cdot)$ is obviously finite.

	Naturally, for any real number $p$, the $L_p$ Gaussian surface area measure, as the $L_p$ generalization of \eqref{Gmeasure}, can be defined as
	\begin{equation}
		S_{p,\gamma^{n}}(K,\eta)=\frac{1}{(\sqrt{2 \pi})^{n}} \int_{\nu_{K}^{-1}(\eta)}|\langle x,\nu_{K}(x) \rangle|^{1-p} e^{-\frac{|x|^{2}}{2}} d \mathcal{H}^{n-1}(x).
	\end{equation}
	In particular, the cases $p=1$ and $p=0$ correspond to the Gaussian surface area measure and the Gaussian cone measure, respectively, as introduced in our previous work \cite{shan2025}. 
	
	It is worth noting that, in the case of convex bodies, the corresponding Gaussian Minkowski problem was studied in \cite{Huang2021}, while its $L_p$ version was provided in \cite{Liu2022}. Additional results on the Gaussian Minkowski problem for convex bodies can be found in  \cite{zhao2023, FHX2023,FLX2023,KLweightedprojection,Liv2019}.
	
	In this paper, we consider the (normalized) $L_{p}$ Gaussian Minkowski problem for $C$-pseudo-cones. For the existence theorem, we discuss two cases: $p > 0$ and $p < 0$. When $p > 0$, the result corresponds to a generalization of the case $p = 1$ in \cite{shan2025}, but when $p < 0$, the situation is completely different. For $p > 0$, the corresponding result is
	\begin{theorem}
		If $p>0$, 
		for any nonzero finite Borel measure  $\mu$ on $\Omega_{C^{\circ}}$,  there exists a $C$-pseudo-cone $K$ such that
		$$\mu=\frac{c}{p}S_{p,\gamma^{n}}(K,\cdot),$$
		where $c=\frac{\int_{\Omega_{C^{\circ}}}\bar{h}_{K}^{p}d\mu}{\gamma^{n}(K)}$.
		
	\end{theorem}
	Here $\gamma^n(K)$ denotes the $n$-dimensional Gaussian measure of $K$ and 	 $\bar{h}_K$ denotes the absolute support function of $K$.
For the case $p<0$, the existence of solutions to the (normalized) $L_{p}$ Gaussian Minkowski problem can be stated as follows.
	
	\begin{theorem}
		If $p<0$, 
		for any nonzero finite Borel measure  $\mu$ on $\Omega_{C^{\circ}}$,  there exists a $C$-pseudo-cone $K$ such that
		$$\mu=-\frac{c}{p}S_{p,\gamma^{n}}(K,\cdot),$$
		where $c=\frac{\int_{\Omega_{C^{\circ}}}\bar{h}_{K}^{p}d\mu}{\gamma^{n}(K)}$.
		
	\end{theorem}

	For the case $p=0$, the limiting case of $L_{p}$ theory, referred to as the Gaussian logarithmic Minkowski problem, has been  studied in \cite{shan2025}.

	The class $\mathcal{K}(C,\omega)$ consists of $C$-pseudo-cones whose construction is determined by the compact set $\omega\subset\Omega_{C^{\circ}}$. Specifically, $K\in\mathcal{K}(C,\omega)$ implies 
	\[
	K = C \cap \bigcap_{u\in\omega} \{ x \in\mathbb{R}^{n}: \langle x, u \rangle \leq h_{K}(u) \},
	\]
where $h_K$ denotes the  support function of $K$.	 Under volume constraints, elements in $\mathcal{K}(C,\omega)$ satisfy the following uniqueness property:
	
	\begin{theorem}
		For $p\in (0,1]$, 	let $\omega\subset\Omega_{C^{\circ}}$ be a nonempty compact set. If	 $K,L\in\mathcal{K}(C,\omega)$ have the same $L_{p}$ Gaussian surface area measure, i.e.,
		$$S_{p,\gamma^{n}}(K,\cdot)=S_{p,\gamma^{n}}(L,\cdot),$$
		and if $\gamma^{n}(K)=\gamma^{n}(L)$,	 then $K=L$.
	\end{theorem}
	
	For certain directional hyperplanes $H_{t}$, the Gaussian surface area density of $C\cap H_t$ tends to zero as $t\to 0^{+}$ and $t\to +\infty$, where $t$ denotes the distance from the origin to the hyperplane. This asymptotic behavior implies that for any cone $C$ and any compact set $\omega\subset\Omega_{C^{\circ}}$, one can always construct examples in $\mathcal{K}(C,\omega)$ such that the solution to the $L_{p}$ Gaussian Minkowski problem loses uniqueness:
	\begin{theorem}
	If $p<n$,	for any pointed, $n$-dimensional closed convex cone $C$ in $\mathbb{R}^{n}$ and any nonempty compact set $\omega\subset\Omega_{C^{\circ}}$, 	there exist distinct sets $K, L \in \mathcal{K}(C, \omega)$ such that their $L_p$ Gaussian surface area measures coincide, i.e., $S_{p,\gamma^{n}}(K,\cdot)=S_{p,\gamma^{n}}(L,\cdot)$, but $K\neq L$.
	\end{theorem}
	
	\section{Preliminaries}\label{section2}
	Let $\left( \mathbb{R}^n, \langle \cdot, \cdot \rangle\right) $ denote the $n$-dimensional Euclidean space  with the standard inner product. A subset $C \subset \mathbb{R}^n$ is called a \textit{cone} if it satisfies $\lambda x \in C$ for all $x \in C$ and $\lambda \geq 0$. A closed convex set $K \subset \mathbb{R}^n$ is termed a \textit{pseudo-cone} if it satisfies $\lambda x \in K$ for all $x \in K$ and $\lambda \geq 1$. The recession cone of $K$, characterizing its asymptotic behavior, is defined as $\operatorname{rec} K = \{ z \in \mathbb{R}^n : K + z \subseteq K \}$. For a pointed closed convex cone $C \subset \mathbb{R}^n$, a \textit{$C$-pseudo-cone} refers to a pseudo-cone with $\operatorname{rec} K = C$. The dual cone is given by $C^\circ = \{ x \in \mathbb{R}^n : \langle x, y \rangle \leq 0 \ \forall y \in C \}$. Throughout this paper, we assume that any cone $C$ is pointed (i.e., it contains no lines) and $n$-dimensional.
	
	Let $S^{n-1}$ represent the unit sphere and $B$ the open unit ball.
	For $u \in S^{n-1}$ and $t \in \mathbb{R}$, we define $H(u, t) = \{ x \in \mathbb{R}^n : \langle x, u \rangle = t \}$ and its associated halfspace 
	$$H^-(u, t) = \{ x \in \mathbb{R}^n : \langle x, u \rangle \leq t \},\quad H^{+}(u, t) = \{ x \in \mathbb{R}^n : \langle x, u \rangle \geq t \}.$$ 
	Fixing a unit vector $\mathfrak{v} \in \operatorname{int} C \cap \operatorname{int}(-C^\circ)$, we partition $C$ via $$C^-(t) = C \cap H^-(\mathfrak{v}, t), C^+(t) = C \cap H^+(\mathfrak{v}, t), \quad\text{and}\quad C(t) = C \cap H(\mathfrak{v}, t)$$ for $t > 0$. Correspondingly, any $C$-pseudo-cone $K$ is decomposed into $K(t) = K \cap C(t)$, $K^-(t) = K \cap C^-(t)$, and $K^+(t) = K \cap C^+(t)$.

For a $C$-pseudo-cone $K$, the \textit{support function} $h_K: C^\circ \to \mathbb{R}$, defined as $$h_K(x) = \sup \{ \langle x, y \rangle : y \in K \},$$ and it is bounded and non-positive. Its absolute support function is given by $$\bar{h}_K = -h_K.$$ 
The radial function $\varrho_K: \Omega_C \to \mathbb{R}$, where $	\Omega_C = S^{n-1} \cap \operatorname{int} C,$ given by $$\varrho_K(v) = \min \{ \lambda \in \mathbb{R} : \lambda v \in K \},$$ generates the radial map $r_K = \varrho_K(v)v$ and radial Gauss map $\alpha_K = \nu_K \circ r_K$, where $\nu_K$ denotes the outer unit normal  on $\partial K$.

	Convergence of $C$-pseudo-cones follows Schneider \cite{schneider pescone}: 
	\begin{definition}\label{conv def}
	A sequence $K_i$ converges to $K$ if there exists $t_0 > 0$ such that $K_i^-(t_0) \neq \emptyset$ for all $i$ and $K_i^-(t) \to K^-(t)$ in Hausdorff metric for $t \geq t_0$. 
	\end{definition}
	
	The following selection theorem of $C$-pseudo-cones is essential: 
	\begin{lemma}\label{select}
	For a sequence $K_i$ with uniform bounds $a < \operatorname{dist}(o, K_i) < b$, there exists a convergent subsequence $K_{i_j} \to K$. 
	\end{lemma}

	An integral transformation formula for $C$-pseudo-cones states that \cite{schneider weighted}: 
	
	\begin{lemma}\label{integral trans}
		Let $K$ be a $C$-pseudo-cone, for any nonnegative and Borel measurable or $\mathcal{H}^{n-1}$-integrable function $F: \partial_{i}K\to \mathbb{R}$:
		\begin{align*}
			\int_{\partial_{i} K} F(y) d\mathcal{H}^{n-1}(y) & =\int_{\Omega_{C}} F\left(r_{K}(v)\right) \frac{\varrho_{K}^{n}(v)}{\bar{h}_{K}\left(\alpha_{K}(v)\right)} d v \\
			& =\int_{\Omega_{C}} F\left(r_{K}(v)\right) \frac{\varrho_{K}^{n-1}(v)}{\left|\left\langle v, \alpha_{K}(v)\right\rangle\right|} d v,
		\end{align*}
		where $\partial_{i} K:=\partial K\cap \operatorname{int} C$.
	\end{lemma}

		In this paper, we define $\Omega := \Omega_{C^\circ} = S^{n-1} \cap \operatorname{int} C^\circ$. For any $u \in \Omega$, the support hyperplane of a $C$-pseudo-cone $K$ is defined as
	$$
	H_{K}(u)=\left\{x \in \mathbb{R}^{n}: \langle x, u\rangle=h_{K}(u)\right\}, \quad H^{-}_{K}(u)=\left\{x \in \mathbb{R}^{n}: \langle x, u\rangle \leq h_{K}(u)\right\}.$$
	
	A $C$-pseudo-cone $K$ is \textit{$C$-determined} by a compact set $\omega \subset \Omega = S^{n-1} \cap \operatorname{int} C^\circ$ if $K = C \cap \bigcap_{u \in \omega} H_K^-(u)$. The collection $\mathcal{K}(C, \omega)$ consists of such $C$-pseudo-cones that are $C$-determined by $\omega$.  
	
	 We denote $$K^{(\omega)} = C \cap \bigcap_{u \in \omega} H_K^-(u)$$
	 for a $C$-pseudo-cone $K$. $K^{(\omega)}\in\mathcal{K}(C, \omega)$.
	
	The Wulff shape construction provides critical examples: For compact $\omega \subset \Omega$ and positive continuous function $h: \omega \to \mathbb{R}_+$, the   \textit{Wulff shape} $[h]$ associated with $(C,\omega,h)$ is defined as:
	\[
	[h] = C \cap \bigcap_{u \in \omega} \{ y \in \mathbb{R}^n : \langle y, u \rangle \leq -h(u) \}
	\]
	belongs to $\mathcal{K}(C, \omega)$.

	The following lemma, established in \cite{schneider weighted}, will be used in our analysis:
	\begin{lemma}\label{compace conv}
		Let $K_j$ be a sequence of $C$-pseudo-cones converging to $K$, i.e., $K_j \to K$. For any nonempty compact set $\omega \subset \Omega$, then $K_j^{(\omega)}$ converge to $K^{(\omega)}$.
	\end{lemma}

Let $E$ be a convex set in $\mathbb{R}^{n}$. Its $n$-dimensional Gaussian measure $\gamma^n(E)$ is given by
$$\gamma^{n}(E):=\frac{1}{(2 \pi)^{\frac{n}{2}}}\int_{E} e^{-\frac{|x|^{2}}{2}} d x.$$

	The Gaussian measure $\gamma^{n}$ is log-concave with respect to convex sets, that is:
		\begin{equation*}
			\gamma^{n}((1-t) K+t L) \geq \gamma^{n}(K)^{1-t} \gamma^{n}(L)^{t}
		\end{equation*}
		equality holds if and only if $K=L$.
For further inequalities related to the Gaussian measure, we refer to
	\cite{cianchideficit,Z2021,Gardner2010,milman2005,KLweightedprojection,
		Lutwak1993,Lutwak1996}. For a comprehensive treatment of the theory of convex bodies, the classical texts \cite{Gardnerbook2006} and \cite{schneiderbook2014} are recommended references.
	\\

	\section{The $L_{p}$ Gaussian surface area measure}
	The primary aim of this section is to establish the $L_{p}$ variational formula. For this purpose, we need the following lemma from  Schneider \cite{schneider weighted}. Let $ \Omega := \Omega_{C^\circ} = S^{n-1} \cap \operatorname{int} C^\circ$.
	\begin{lemma}\label{rho deri}
		Let $ \omega \subset \Omega$  be a nonempty  compact set, and  a $C$-pseudo-cone $K \in \mathcal{K}(C, \omega)$. Given a continuous function  $f: \omega \rightarrow \mathbb{R}$,  then there exists a   constant  $\delta>0$ which is small enough  such that for any $|t| \leq \delta$, the function  $h_{t}$  defined as
		\begin{equation}\label{logminkowskisum}
			\log h_{t}(u)=\log \bar{h}_{K}(u)+t f(u)+o(t, u), \quad u \in \omega,
		\end{equation}
		 where   $o(t, \cdot): \omega \rightarrow \mathbb{R}$  is a continuous function satisfying  $\lim _{t \rightarrow 0} o(t, \cdot) / t=0$  uniformly on  $\omega$. The Wulff shape $\left[h_{t}\right]$   associated with  $\left(C, \omega, h_{t}\right)$ is termed  \textit{a logarithmic
			family of Wulff shapes}.
		
		For almost all  $v \in \Omega_{C}$, the derivative of the radial function is given by
		\begin{equation}\label{log rho}
			\left.\frac{\mathrm{d} \varrho_{\left[h_{t}\right]}(v)}{\mathrm{d} t}\right|_{t=0}=\lim _{t \rightarrow 0} \frac{\varrho_{\left[h_{t}\right]}(v)-\varrho_{K}(v)}{t}=f\left(\alpha_{K}(v)\right)\varrho_{K}(v).
		\end{equation}
	\end{lemma}
	
	The Gaussian covolume $V_{G}(K) $ of a $C$-pseudo-cone $K$ is given  in \cite{shan2025} by
	\begin{definition}\label{coco}
		$$V_{G}(K)=\gamma^{n}(C \backslash K).$$
	\end{definition}

	Recall the \textit{$L_{p}$ Gaussian surface area measure} $S_{p,\gamma^{n}}(K,\cdot)$ of a $C$-pseudo-cone $K$ is  defined as follows:
	\begin{definition}
		For every Borel set $\eta\subset\Omega$,
		\begin{equation*}
			S_{p,\gamma^{n}}(K,\eta):=\frac{1}{(\sqrt{2 \pi})^{n}} \int_{\nu_{K}^{-1}(\eta)}|\langle x,\nu_{K}(x) \rangle|^{1-p} e^{-\frac{|x|^{2}}{2}} d \mathcal{H}^{n-1}(x)=\int_{\eta}\bar{h}_{K}^{1-p}(u)d	S_{\gamma^{n}}(K,u).
		\end{equation*}
	
	\end{definition}

	The following $L_{p}$ variational formula serves as the foundation for solving the $L_{p}$ Minkowski problem using variational methods. We begin by proving the variational formula for the Gaussian covolume. The variational approach originated in \cite{HLYZdual}.
	
	\begin{lemma}\label{covari}
		If  $K \in \mathcal{K}(C, \omega)$ for some nonempty compact set  $\omega \subset \Omega$.  For $p\neq 0$,  let  $f: \omega \rightarrow \mathbb{R}$ be a continuous function,  and  $\left[\left( \bar{h}_{K}(u)^{p}+t f(u)\right)^{\frac{1}{p}}\right]$   denotes the  Wulff shapes associated with  $\left(C, \omega, \left( \bar{h}_{K}(u)^{p}+t f(u)\right)^{\frac{1}{p}}\right) $, then
		\begin{equation}
			\lim _{t \rightarrow 0} \frac{V_{G}\left(\left[\left( \bar{h}_{K}(u)^{p}+t f(u)\right)^{\frac{1}{p}}\right]\right)-V_{G}(K)}{t}=\frac{1}{p}\int_{\omega} f(u) dS_{p,\gamma^{n}}(K,  u).
		\end{equation}
		
	\end{lemma}
	
	\begin{proof}
	For $p\neq 0$, let
		$$h_{t}(u)=\left( \bar{h}_{K}(u)^{p}+t f(u)\right)^{\frac{1}{p}},$$
		then
		\begin{equation}\label{logp}
		\log h_{t}=\log\bar{h}_{K}+\frac{tf}{p\bar{h}_{K}^{p}}+o(t,u),
		\end{equation}
		 where   $o(t, \cdot): \omega \rightarrow \mathbb{R}$  is a continuous function satisfying  $\lim _{t \rightarrow 0} o(t, \cdot) / t=0$  uniformly on  $\omega$.
		 
		Using polar coordinates, we have
		$$V_{G}\left(\left[h_{t}\right]\right)=\frac{1}{(\sqrt{2 \pi})^{n}}\int_{\Omega_{C}} \int_{0}^{\varrho_{\left[h_{t}\right]}(v)} e^{-\frac{r^{2}}{2}} r^{n-1} dr d v=\frac{1}{(\sqrt{2 \pi})^{n}}\int_{\Omega_{C}} F_{t}(v)dv,$$
		where we denote  $F_{t}(v)=\int_{0}^{\varrho_{\left[h_{t}\right]}(v)} e^{-\frac{r^{2}}{2}} r^{n-1} dr$.
		Therefore,
		\begin{align}
			\frac{F_{t}(v)-F_{0}(v)}{t} & =\frac{1}{t} \int_{\varrho_{K}(v)}^{\varrho_{\left[h_{t}\right]}(v)} e^{-\frac{r^{2}}{2}} r^{n-1} d r \\
			& =\frac{\varrho_{\left[h_{t}\right]}(v)-\varrho_{K}(v)}{t} \cdot \frac{1}{\varrho_{\left[h_{t}\right]}(v)-\varrho_{K}(v)} \int_{\varrho_{K}(v)}^{\varrho_{\left[h_{t}\right]}(v)} e^{-\frac{r^{2}}{2}} r^{n-1} d r.
		\end{align}
		Moreover, the last term converges to $e^{-\frac{\varrho_{K}^{2}}{2}} \varrho_{K}^{n-1}$. By Lemma \ref{rho deri}, we have 
			\begin{equation}
			\left.\frac{\mathrm{d} \varrho_{\left[h_{t}\right]}(v)}{\mathrm{d} t}\right|_{t=0}=\lim _{t \rightarrow 0} \frac{\varrho_{\left[h_{t}\right]}(v)-\varrho_{K}(v)}{t}=\frac{f\left(\alpha_{K}(v)\right)\varrho_{K}(v)}{p\bar{h}_{K}^{p}(\alpha_{K}(v))},
		\end{equation}
		it follows that
		\begin{equation}
			\lim _{t \rightarrow 0} \frac{F_{t}(v)-F_{0}(v)}{t}=e^{-\frac{\varrho_{K}^{2}}{2}} \frac{f\left(\alpha_{K}(v)\right) \varrho_{K}^{n}(v)}{p\bar{h}_{K}^{p}(\alpha_{K}(v))}
		\end{equation}
		for almost  $v \in \Omega_{C}$. Moreover, for sufficiently small $|t|$, there exists a constant $M$ satisfying
		$$\left|\frac{F_{t}(v)-F_{0}(v)}{t}\right|\le M.$$
	Applying the dominated convergence theorem and Lemma \ref{integral trans},  we obtain 
		\begin{align*}
			\lim _{t \rightarrow 0} \frac{V_{G}\left(\left[h_{t}\right]\right)-V_{G}(K)}{t}&=\frac{1}{(\sqrt{2 \pi})^{n}}\int_{\Omega_{C}}	\lim _{t \rightarrow 0} \frac{F_{t}(v)-F_{0}(v)}{t} dv\\
			&=\frac{1}{(\sqrt{2 \pi})^{n}}\int_{\Omega_{C}}e^{-\frac{\varrho_{K}^{2}}{2}} \frac{f\left(\alpha_{K}(v)\right) \varrho_{K}^{n}(v)}{p\bar{h}_{K}^{p}(\alpha_{K}(v))} dv\\
			&=\int_{\Omega}f(u)\bar{h}_{K}(u)^{1-p}dS_{\gamma^{n}}(K,  u)\\
			&=\int_{\omega}f(u)dS_{p,\gamma^{n}}(K,  u).
		\end{align*}
	\end{proof}
	
	Naturally, the $L_{p}$ variational formula for the Gaussian volume is as follows.
	\begin{lemma}\label{vari formula}
		If  $K \in \mathcal{K}(C, \omega)$ for some nonempty compact set  $\omega \subset \Omega$.  For $p\neq 0$,  let  $f: \omega \rightarrow \mathbb{R}$ be a continuous function,  and  $\left[\left( \bar{h}_{K}(u)^{p}+t f(u)\right)^{\frac{1}{p}}\right]$   denotes the  Wulff shapes associated with  $\left(C, \omega, \left( \bar{h}_{K}(u)^{p}+t f(u)\right)^{\frac{1}{p}}\right) $, then
	\begin{equation}
		\lim _{t \rightarrow 0} \frac{\gamma^{n}\left(\left[\left( \bar{h}_{K}(u)^{p}+t f(u)\right)^{\frac{1}{p}}\right]\right)-\gamma^{n}(K)}{t}=-\frac{1}{p}\int_{\omega} f(u) dS_{p,\gamma^{n}}(K,  u).
	\end{equation}
		
	\end{lemma}
	\begin{proof}
		This follows from Lemma \ref{covari} and the fact
		$$\gamma^{n}(\left[\left( \bar{h}_{K}(u)^{p}+t f(u)\right)^{\frac{1}{p}}\right])+V_{G}(\left[\left( \bar{h}_{K}(u)^{p}+t f(u)\right)^{\frac{1}{p}}\right])=\gamma^{n}(K)+V_{G}(K)=\gamma^{n}(C).$$
	\end{proof}
	
	By the definition of $S_{p,\gamma^{n}}K$ and the results established in our previous work \cite{shan2025} regarding the finiteness and weak convergence of the Gaussian surface area measure ($p=1$), we conclude that the $L_{p}$ Gaussian surface area measure is also a finite measure. Furthermore, it is weakly convergent in the sense of Definition \ref{conv def} when restricted to the set $\mathcal{K}(C, \omega)$, for some nonempty compact set $\omega\subset \Omega$:
	
	\begin{lemma}\label{weal  conve}
		For a nonempty compact set  $\omega \subset \Omega$,   and  $ K_{j} \in \mathcal{K}(C, \omega)$,  $j \in   \mathbb{N}$. If $K_{j} \rightarrow K $ as  $j \rightarrow \infty $,  then the measure $S_{p,\gamma^{n}}\left(K_{j}, \cdot\right)$ converge weakly to $ S_{p,\gamma^{n}}\left(K, \cdot\right)$.
	\end{lemma}
	
	\begin{proof}
		This follows from  the support functions $\bar{h}_{K_{j}}$  converge uniformly to $\bar{h}_{K}$ and the weak convergence of the Gaussian surface area measure.
	\end{proof}

\section{Existence of solutions when $p>0$}
We first need to establish a uniform estimate for the decay properties of the Gaussian measure.
\begin{lemma}\label{decayesti}
	$$\gamma^{n}(\mathbb{R}^{n}-rB)\le\frac{2n\sqrt{n}}{(2 \pi)^{\frac{1}{2}}r}e^{-\frac{r^{2}}{2n}}.$$
\end{lemma}

\begin{proof}If $t>0$, 
	\begin{align*}
		\gamma^{1}(x>t)
		&=\frac{1}{(2 \pi)^{\frac{1}{2}}}\int_{t}^{+\infty} e^{-\frac{r^{2}}{2}} dr\\
		&=\frac{1}{(2 \pi)^{\frac{1}{2}}}\int_{t}^{+\infty}  \frac{1}{r}\cdot re^{-\frac{r^{2}}{2}} dr\\
		&=\frac{1}{(2 \pi)^{\frac{1}{2}}}\int_{t}^{+\infty}  -\frac{1}{r} d(e^{-\frac{r^{2}}{2}})\\
		&=\frac{1}{(2 \pi)^{\frac{1}{2}}}\left( -\frac{1}{r} e^{-\frac{r^{2}}{2}} \Big|_{t}^{+\infty}   -  \int_{t}^{+\infty}  \frac{1}{r^{2}} e^{-\frac{r^{2}}{2}}dr       \right) \\
		&=\frac{1}{(2 \pi)^{\frac{1}{2}}}\left( \frac{1}{t} e^{-\frac{t^{2}}{2}}    -  \int_{t}^{+\infty}  \frac{1}{r^{2}} e^{-\frac{r^{2}}{2}}dr       \right) \\
		&\le \frac{1}{(2 \pi)^{\frac{1}{2}}}\cdot \frac{1}{t} e^{-\frac{t^{2}}{2}}.  
	\end{align*}
	For a ball $rB$ of radius $r$ in $\mathbb{R}^{n}$, let $P_{r}$ be its inscribed cube. Then, the edge length of $P_{r}$ is $\frac{2r}{\sqrt{n}}$. Since $P_{r}\subset rB$, then $\mathbb{R}^{n}-rB \subset \mathbb{R}^{n}-P_{r}$, we conclude
	\begin{align*}
		\gamma^{n}(\mathbb{R}^{n}-rB)&\le \gamma^{n}(\mathbb{R}^{n}-P_{r})\\
		&\le \gamma^{n}(|x_{1}|>\frac{r}{\sqrt{n}})+\gamma^{n}(|x_{2}|>\frac{r}{\sqrt{n}})+\cdots+\gamma^{n}(|x_{n}|>\frac{r}{\sqrt{n}})\\
		&=2n\gamma^{1}(x>\frac{r}{\sqrt{n}})\\
		&\le 2n \cdot \frac{1}{(2 \pi)^{\frac{1}{2}}}\cdot \frac{1}{\frac{r}{\sqrt{n}}} e^{-\frac{\left(\frac{r}{\sqrt{n}}\right)  ^{2}}{2}}\\
		&=\frac{2n\sqrt{n}}{(2 \pi)^{\frac{1}{2}}r}e^{-\frac{r^{2}}{2n}}.
	\end{align*}
\end{proof}

Next, we hope to transform the $L_{p}$ Gaussian Minkowski problem into an optimization problem. When $p > 0$, the method for solving this problem is similar to \cite{shan2025}, but
with modifications.

We assume $\omega\subset\Omega$ is a nonempty compact subset.
For any nonzero finite Borel measure  $\mu$ on $\omega$,
we define the set $C^{+}(\omega)$ as the collection of continuous functions $f: \omega \to (0, \infty)$. If $p>0$, we introduce a functional $I_{\mu}: C^{+}(\omega) \to (0, \infty)$ given by
$$
I_{\mu}(f) := \gamma^{n}([f]) \int_{\omega} f^{p}  d\mu
$$
for each $f \in C^{+}(\omega)$, where $[f]$ represents the Wulff shape associated with  $(C, \omega, f)$. 

Given $p>0$, the following optimization problem is then considered:
$$\sup \{I_{\mu}(f)=\gamma^{n}([f])\int_{\omega} f^{p} d\mu :f\in C^{+}(\omega)\}.$$ 
Obviously, for any $K\in \mathcal{K}(C,\omega)$, $I_{\mu}(\bar{h}_{K})=\gamma^{n}(K)\int_{\omega} \bar{h}_{K}^{p} d\mu>0$.

If  $I_{\mu}$  attains a maximum at $\bar{h}_{K}$ for some $K\in \mathcal{K}(C,\omega)$, 
for any continuous function $f:\omega\rightarrow \mathbb{R}$,   the function $\left( \bar{h}_{K}^{p}+tf\right) ^{\frac{1}{p}}$ belongs to $C^{+}(\omega)$ for sufficiently small $|t|$.  By Lemma \ref{vari formula}, we obtain
\begin{equation*}
	0=\left.\frac{d}{d t}\right|_{t=0}I_{\mu}(\left( \bar{h}_{K}^{p}+tf\right) ^{\frac{1}{p}})= \gamma^{n}(K)\int_{\omega}fd\mu-\frac{1}{p}\int_{\omega}fdS_{p,\gamma^{n}}(K,\cdot)\int_{\omega}\bar{h}_{K}^{p}d\mu.
\end{equation*}
In other words, 
\begin{equation*}
	\gamma^{n}(K)\int_{\omega}fd\mu=\frac{1}{p}\int_{\omega}fdS_{p,\gamma^{n}}(K,\cdot)\int_{\omega}\bar{h}_{K}^{p}d\mu
\end{equation*}
holds for all continuous functions $f$ on $\omega$. Then we conclude that
\begin{equation}\label{com mink}
\mu=\frac{\int_{\omega}\bar{h}_{K}^{p}d\mu}{p\gamma^{n}(K)}S_{p,\gamma^{n}}(K,\cdot).
\end{equation}
Therefore, in what follows, we will show that $I_{\mu}$ indeed attains its maximum at  support function $\bar{h}_{K}$ for some $K\in \mathcal{K}(C,\omega)$. From the preceding discussion, this is equivalent to $K$ being a solution to the (normalized) $L_{p}$ Gaussian Minkowski problem.
\begin{lemma}\label{compact cone minkowski}
If $p>0$, suppose $\omega\subset\Omega$ is a nonempty compact subset.
For any nonzero finite Borel measure  $\mu$ on $\omega$,  there exists a $C$-pseudo-cone $K\in\mathcal{K}(C,\omega)$ such that
$$\mu=\frac{\int_{\omega}\bar{h}_{K}^{p}d\mu}{p\gamma^{n}(K)}S_{p,\gamma^{n}}(K,\cdot).$$
	
\end{lemma}
\begin{proof}
	
	For any $f\in C^{+}(\omega)$,  we have $\bar{h}_{[f]}\ge f$. Consequently,
	$$I_{\mu}(f)=\gamma^{n}([f])\int_{\omega} f^{p} d\mu\le \gamma^{n}([f])\int_{\omega}\bar{h}_{[f]}^{p} d\mu=I_{\mu}(\bar{h}_{[f]} ).$$
	This  shows that the supremum of $I_{\mu}(f)$ is attained by the support functions of $\mathcal{K}(C,\omega)$.
	Given a sequence  $\{\bar{h}_{K_{i}}\}$, where each $K_i \in \mathcal{K}(C, \omega)$, satisfying
$$\lim _{i \rightarrow \infty} I_{\mu}(\bar{h}_{K_{i}})=\sup \left\{I_{\mu}(f): f \in C^{+}\left(\omega\right)\right\}.$$
Define
$$r_{i}=\min\{r:rB\cap K_{i}\neq\emptyset \},$$
then we have
$$\bar{h}_{K_{i}}\le r_{i}.$$
and
$$K_{i}\subset \mathbb{R}^{n}-r_{i}B.$$
If there exists a subsequence, still denoted by $r_{i}$, such that $r_{i}\rightarrow\infty$,   this leads to the estimate by Lemma \ref{decayesti}:
\begin{align*}
	I_{\mu}(\bar{h}_{K_{i}})&=\gamma^{n}(K_{i})\int_{\omega}\bar{h}_{K_{i}}^{p}  d\mu\\
	&\le \gamma^{n}(\mathbb{R}^{n}-r_{i}B)\mu(\omega) r_{i}^{p}\\
	&\le\frac{c}{r_{i}}e^{-\frac{r_{i}^{2}}{2n}}\cdot r_{i}^{p}\rightarrow 0,
\end{align*}
for some constant $c$. This contradicts the fact that $\lim _{i \rightarrow \infty} I_{\mu}(\bar{h}_{K_{i}})=\sup I_{\mu}$, since $I_{\mu}(\bar{h}_{K})>0$ for each $K\in \mathcal{K}(C,\omega)$.

Moreover, assume that $r_{i}\rightarrow 0$,
then we have
$$\gamma^{n}(\mathbb{R}^{n}-r_{i}B)\rightarrow 1,$$
and through  the same estimate, we obtain
\begin{align*}
	I_{\mu}(\bar{h}_{K_{i}})&=\gamma^{n}(K_{i})\int_{\omega}\bar{h}_{K_{i}}^{p}  d\mu\\
	&\le\gamma^{n}(\mathbb{R}^{n}-r_{i}B)\mu(\omega) r_{i}^{p}\rightarrow 0.
\end{align*}
Similarly, we arrive at a contradiction.

This implies the uniform estimate $0<c_{1}<\operatorname{dist}(o,\partial K_{i})<c_{2}$, for some constants $c_{1}$ and $c_{2}$. From Lemma \ref{select}, we can find a  subsequence $K_{i_{j}}$, such that $K_{i_{j}}\rightarrow K$
for some $K\in\mathcal{K}(C,\omega)$. Consequently, we get $\lim _{i \rightarrow \infty} I_{\mu}(\bar{h}_{K_{i_{j}}})=I_{\mu}(\bar{h}_{K})=\sup I_{\mu}(f)$.   By the preceding discussion \eqref{com mink},
the result is proved.

\end{proof}

If $\mu$ is a nonzero Borel measure on $\Omega$, $p>0$, we may approximate $\Omega$ by an increasing sequence of compact sets $\{\omega_i\}$, i.e., $\omega_i \subset \operatorname{int}\omega_{i+1}$ and $\bigcup_{i \in \mathbb{N}} \omega_i = \Omega$. For each $i \in \mathbb{N}$, denote  $\mu_i = \mu \llcorner \omega_i$ as 
\[
\mu \llcorner \omega_i(\sigma) := \mu(\sigma \cap \omega_i)
\]  
for any Borel subset $\sigma \subset \Omega$. We further require that $\mu_0 = \mu \llcorner \omega_0$ is a nonzero measure. By Lemma \ref{compact cone minkowski}, for each $i \in \mathbb{N}$, there exists a $C$-pseudo-cone $K_i \in \mathcal{K}(C, \omega_i)$ satisfying  
\begin{equation}\label{Ki}
\mu_{i}=\frac{\int_{\omega_{i}}\bar{h}_{K_{i}}^{p}d\mu_{i}}{p\gamma^{n}(K_{i})}S_{p,\gamma^{n}}(K_{i},\cdot).
\end{equation}
To establish the existence of  solutions to the $L_{p}$ Gaussian Minkowski problem on $\Omega$, it remains to show that the sets $\{K_i\}$ are uniformly bounded, i.e., there exist positive uniform lower and upper bounds for $  \operatorname{dist}(o, \partial K_i) $. This uniform boundedness will allow us to extract a convergent subsequence, ultimately yielding the desired solution on $\Omega$ via an approximation argument.
The  uniform estimate as follows.

\begin{lemma}\label{uniesti}
For the sequence $\{K_{i}\}$ in \eqref{Ki}, there exist two constant $m$ and $M$ such that
$$0<m<\operatorname{dist}(o, \partial K_i) <M$$
for all $i\in\mathbb{N}$.
\end{lemma}
\begin{proof}
	Recall the optimization problems 
	$$I_{\mu_{i}}(\bar{h}_{K_{i}})=\sup \left\{I_{\mu_{i}}(f)=\gamma^{n}([f])\int_{\omega_{i}} f^{p} d\mu_{i}: f \in C^{+}\left(\omega_{i}\right)\right\}.$$
	Since $\mu_{i}\le\mu_{i+1}\le\mu$, $\omega_{i}\subset \operatorname{int}\omega_{i+1}$ and $\mathcal{K}(C,\omega_{i})\subset\mathcal{K}(C,\omega_{i+1})$, we get
\begin{align*}
	0<I_{\mu_{i}}(\bar{h}_{K_{i}})=&\gamma^{n}([\bar{h}_{K_{i}}|_{\omega_{i}}])\int_{\omega_{i}}\bar{h}_{K_{i}}^{p}  d\mu_{i}\\
	&=\gamma^{n}([\bar{h}_{K_{i}}|_{\omega_{i+1}}])\int_{\omega_{i}}\bar{h}_{K_{i}}^{p}  d\mu_{i}\\
	&\le \gamma^{n}([\bar{h}_{K_{i}}|_{\omega_{i+1}}])\int_{\omega_{i+1}}\bar{h}_{K_{i}}^{p}  d\mu_{i+1}\\&\le I_{\mu_{i+1}}(\bar{h}_{K_{i+1}}).
\end{align*}
If we set 
$$r_{i}=\min\{r:rB\cap K_{i}\neq\emptyset \},$$
  Then we also have
$$\bar{h}_{K_{i}}\le r_{i}\quad\text{and}\quad K_{i}\subset \mathbb{R}^{n}-r_{i}B$$
as Lemma \ref{compact cone minkowski}. If $r_{i}\rightarrow\infty$,  	 then
\begin{align*}
	I_{\mu_{i}}(\bar{h}_{K_{i}})=&\gamma^{n}(K_{i})\int_{\omega_{i}}\bar{h}_{K_{i}}^{p}  d\mu_{i}\\
	&\le \gamma^{n}(\mathbb{R}^{n}-r_{i}B)\mu_{i}(\omega_{i}) r_{i}^{p}\\
	&\le \gamma^{n}(\mathbb{R}^{n}-r_{i}B)\mu(\Omega) r_{i}^{p}\\
	&\le\frac{d}{r_{i}}e^{-\frac{r_{i}^{2}}{2n}}\cdot r_{i}^{p}\rightarrow 0,
\end{align*}
for some constant $d$.
This contradicts the monotonicity property of $I_{\mu_{i}}(\bar{h}_{K_{i}})$, namely  $	0<I_{\mu_{i}}(\bar{h}_{K_{i}})\le I_{\mu_{i+1}}(\bar{h}_{K_{i+1}})\le \cdots$.  Meanwhile, if $r_{i}\rightarrow0$, then
\begin{align*}
	I_{\mu_{i}}(\bar{h}_{K_{i}})&=\gamma^{n}(K_{i})\int_{\omega_{i}}\bar{h}_{K_{i}}^{p}  d\mu_{i}\\
	&\le\gamma^{n}(\mathbb{R}^{n}-r_{i}B)\mu(\Omega) r_{i}^{p}\rightarrow 0,
\end{align*}
this also arrive at a contradiction.
\end{proof}

By the uniform estimate and standard approximation argument, the (normalized) $L_{p}$ Gaussian Minkowski  theorem can be established.
\begin{theorem}\label{ minkowskip>0}
	If $p>0$, 
	for any nonzero finite Borel measure  $\mu$ on $\Omega$,  there exists a $C$-pseudo-cone $K$ such that
	$$\mu=\frac{\int_{\Omega}\bar{h}_{K}^{p}d\mu}{p\gamma^{n}(K)}S_{p,\gamma^{n}}(K,\cdot).$$
	
\end{theorem}

\begin{proof}
	For the sequence $\{K_{i}\}$ in \eqref{Ki} and the uniform estimate of $\operatorname{dist}(o, \partial K_i)$, we can extract a convergent subsequence
	   $K_{i}\rightarrow K$.  
	 
	Fix  $i\in\mathbb{N}$, choose a nonempty compact set $\beta\subset\Omega$ with $\omega_{i}\subset\operatorname{int}\beta$. Recall the set $K^{(\beta)}\in\mathcal{K}(C,\beta)$ is given by $$K^{(\beta)}:=C \cap \bigcap_{u \in \beta}H_{K}^{-}(u).$$
	 By Lemma 13 in \cite{schneider weighted}, we have
	$$\nu_{K_{j}}^{-1}(\omega_{i})=\nu_{K_{j}^{(\beta)}}^{-1}(\omega_{i}),\quad \nu_{K}^{-1}(\omega_{i})=\nu_{K^{(\beta)}}^{-1}(\omega_{i}),$$
	Moreover, by Lemma \ref{compace conv}, it follows that
	$$K_{i}^{(\beta)}\rightarrow K^{(\beta)}.$$
	From Lemma \ref{weal  conve},  the measure $S_{p,\gamma^{n}}(K_{j})\llcorner\omega_{i}$ converges weakly to $S_{p,\gamma^{n}}(K^{(\beta)})\llcorner\omega_{i}$. Consequently, we obtain
	$$c_{j}S_{p,\gamma^{n}}(K_{j},\cdot)\llcorner\omega_{i}\xrightarrow{w}cS_{p,\gamma^{n}}(K,\cdot)\llcorner\omega_{i},$$
	where
	$$c_{i}=\frac{\int_{\omega_{i}}\bar{h}_{K_{i}}^{p}d\mu_{i}}{p\gamma^{n}(K_{i})}\quad\text{and}\quad c=\frac{\int_{\Omega}\bar{h}_{K}^{p}d\mu}{p\gamma^{n}(K)}.$$
	Since $\mu\llcorner\omega_{i}=c_{i}S_{p,\gamma^{n}}(K_{i},\cdot)\llcorner\omega_{i}$,   the equality $$c S_{p,\gamma^{n}}(K,\sigma)=c S_{p,\gamma^{n}}(K^{(\beta)},\sigma)=\mu_{i}(\sigma)=\mu(\sigma)$$ holds  for every Borel subset $\sigma\subset\omega_{i}$. Furthermore,
	since $\cup_{j\in\mathbb{N}} \omega_{j}=\Omega$, we deduce that 
	$cS_{p,\gamma^{n}}(K,\cdot)=\mu$.
\end{proof}

\begin{remark}\label{reason}
	For $p<0$, the corresponding optimization problems 
	\begin{equation}\label{re1}
	\sup \{I_{\mu}(f)=\gamma^{n}([f])\int_{\omega} f^{p} d\mu :f\in C^{+}(\omega)\}
	\end{equation}
	or
		\begin{equation}\label{re2}
		\inf \{I_{\mu}(f)=\gamma^{n}([f])\int_{\omega} f^{p} d\mu :f\in C^{+}(\omega)\}
	\end{equation}
	are both invalid for the following reasons:
	
	(i) For \eqref{re1}, since $p<0$, $f^{p}\ge \bar{h}_{[f]}^{p}$, in other words, $I_{\mu}(f)\ge I_{\mu}(\bar{h}_{[f]})$.
	
	(ii) For \eqref{re2}, if the function $f$ sufficiently big, $f\rightarrow+\infty$, then $\gamma^{n}([f])\rightarrow 0$ and $\int_{\omega} f^{p} d\mu\rightarrow 0$. Therefore, $\inf  I_{\mu}(f)=0$, the variational argument fails.
	
\end{remark}

\section{Existence of solutions when $p<0$}
If $p<0$,  for a nonempty compact subset $\omega\subset\Omega$ and any nonzero Borel measure $\mu$ on $\omega$, unlike the case for $p > 0$ (see Remark \ref{reason}), we consider a new variational functional $\phi_{\mu}: C^{+}(\omega) \to (0, \infty)$ given by
$$
\phi_{\mu}(f) := \frac{\gamma^{n}([f])}{\int_{\omega} f^{p}  d\mu} 
$$
and the following optimization problem:
$$\sup \{\phi_{\mu}(f)=\frac{\gamma^{n}([f])}{\int_{\omega} f^{p}  d\mu} :f\in C^{+}(\omega)\}.$$ 
Obviously, 	for any $f\in C^{+}(\omega)$,  we have $\bar{h}_{[f]}^{p}\le f^{p}$. Then
$$\phi_{\mu}(f)=\frac{\gamma^{n}([f])}{\int_{\omega} f^{p}  d\mu}\le \frac{\gamma^{n}([f])}{\int_{\omega}\bar{h}_{[f]} ^{p}  d\mu}=\phi_{\mu}(\bar{h}_{[f]} ).$$ 
And for each $K\in \mathcal{K}(C,\omega)$, $\phi_{\mu}(\bar{h}_{K})= \frac{\gamma^{n}(K)}{\int_{\omega}\bar{h}_{K} ^{p}  d\mu}>0$.

If  $\phi_{\mu}$  attains a maximum at $\bar{h}_{K}$ for some $K\in \mathcal{K}(C,\omega)$, 
for any continuous function $f:\omega\rightarrow \mathbb{R}$,   the function $\left( \bar{h}_{K}^{p}+tf\right) ^{\frac{1}{p}}$ belongs to $C^{+}(\omega)$ for sufficiently small $|t|$.  By Lemma \ref{vari formula}, we obtain
\begin{equation*}
	0=\left.\frac{d}{d t}\right|_{t=0}\phi_{\mu}(\left( \bar{h}_{K}^{p}+tf\right) ^{\frac{1}{p}})= \dfrac{-\frac{1}{p}\int_{\omega}fdS_{p,\gamma^{n}}(K,\cdot)\int_{\omega}\bar{h}_{K}^{p}d\mu-\gamma^{n}(K)\int_{\omega}fd\mu}{\left( \int_{\omega}\bar{h}_{K} ^{p}  d\mu\right) ^{2}},
\end{equation*}
is equivalent to
\begin{equation*}
	\gamma^{n}(K)\int_{\omega}fd\mu=-\frac{1}{p}\int_{\omega}fdS_{p,\gamma^{n}}(K,\cdot)\int_{\omega}\bar{h}_{K}^{p}d\mu
\end{equation*}
holds for all continuous functions $f$ on $\omega$. Then we obtain
\begin{equation}\label{comp<0}
	\mu=-\frac{1}{p}\frac{\int_{\omega}\bar{h}_{K}^{p}d\mu}{\gamma^{n}(K)}S_{p,\gamma^{n}}(K,\cdot).
\end{equation}
As the case $p>0$,  we will also show that $\phi_{\mu}$ indeed attains its maximum at  support function $\bar{h}_{K}$ for some $K\in \mathcal{K}(C,\omega)$. 
\begin{lemma}\label{minkowski comp<0}
	If $p<0$, suppose $\omega\subset\Omega$ is a nonempty compact subset.
	For any nonzero finite Borel measure  $\mu$ on $\omega$,  there exists a $C$-pseudo-cone $K\in\mathcal{K}(C,\omega)$ such that
	$$\mu=-\frac{\int_{\omega}\bar{h}_{K}^{p}d\mu}{p\gamma^{n}(K)}S_{p,\gamma^{n}}(K,\cdot).$$
	
\end{lemma}
\begin{proof}
	
	For  a maximizing sequence  $\{\bar{h}_{K_{i}}\}$,  $K_i \in \mathcal{K}(C, \omega)$, such that
	$$\lim _{i \rightarrow \infty} \phi_{\mu}(\bar{h}_{K_{i}})=\sup \left\{\phi_{\mu}(f): f \in C^{+}\left(\omega\right)\right\}.$$
	Denote
	$$r_{i}=\min\{r:rB\cap K_{i}\neq\emptyset \},$$
	If  $r_{i}\rightarrow\infty$,  since $\bar{h}_{K_{i}}^{p}\ge r_{i}^{p}$, the following estimate can obtained by Lemma \ref{decayesti}:
	\begin{align*}
		\phi_{\mu}(\bar{h}_{K_{i}})&=\frac{\gamma^{n}(K_{i})}{\int_{\omega}\bar{h}_{K_{i}} ^{p}  d\mu}\\
		&\le\frac{ \gamma^{n}(\mathbb{R}^{n}-r_{i}B)}{\mu(\omega) r_{i}^{p}}\\
		&\le\frac{c}{r_{i}}e^{-\frac{r_{i}^{2}}{2n}}\cdot r_{i}^{-p}\rightarrow 0,
	\end{align*}
	for some constant $c$. This contradicts the fact that $\lim _{i \rightarrow \infty} \phi_{\mu}(\bar{h}_{K_{i}})=\sup \phi_{\mu}>0$.
	
	Moreover, if $r_{i}\rightarrow 0$,
	since
	$$\gamma^{n}(\mathbb{R}^{n}-r_{i}B)\rightarrow 1,$$
	 we get
	\begin{align*}
		\phi_{\mu}(\bar{h}_{K_{i}})&=\frac{\gamma^{n}(K_{i})}{\int_{\omega}\bar{h}_{K_{i}} ^{p}  d\mu}\\
	&\le\frac{ \gamma^{n}(\mathbb{R}^{n}-r_{i}B)}{\mu(\omega) r_{i}^{p}}\rightarrow 0.
	\end{align*}
	We arrive at a contradiction again.
	
	This implies the uniform estimate $0<n_{1}<\operatorname{dist}(o,\partial K_{i})<n_{2}$, for some constants $n_{1}$ and $n_{2}$. From Lemma \ref{select}, we can get a  subsequence $K_{i_{j}}$, such that $K_{i_{j}}\rightarrow K$
	for some $K\in\mathcal{K}(C,\omega)$. Therefore, we have $\lim _{i \rightarrow \infty} \phi_{\mu}(\bar{h}_{K_{i_{j}}})=\phi_{\mu}(\bar{h}_{K})=\sup \phi_{\mu}(f)$.   By \eqref{comp<0},
	the result is proved.

\end{proof}

Let $\mu$ be a nonzero Borel measure on $\Omega$, $p<0$, applying approximation argument again.  Choose an increasing sequence of compact sets $\{\omega_i\}$, such that $\omega_i \subset \operatorname{int}\omega_{i+1}$ and $\bigcup_{i \in \mathbb{N}} \omega_i = \Omega$, where $\omega_{0}$ satisfies   $\mu_0 = \mu \llcorner \omega_0$ nonzero.  By Lemma \ref{minkowski comp<0}, for each $i \in \mathbb{N}$ and $\mu_i = \mu \llcorner \omega_i$, there exists a $C$-pseudo-cone $K_i \in \mathcal{K}(C, \omega_i)$ satisfying  
\begin{equation}\label{Kip<0}
	\mu_{i}=-\frac{\int_{\omega_{i}}\bar{h}_{K_{i}}^{p}d\mu_{i}}{p\gamma^{n}(K_{i})}S_{p,\gamma^{n}}(K_{i},\cdot).
\end{equation}

The  uniform estimate of $K_{i}$ as follows.

\begin{lemma}\label{unip<0}
	For the sequence $\{K_{i}\}$ in \eqref{Kip<0}, there exist two constant $m$ and $M$ such that
	$$0<m<\operatorname{dist}(o, \partial K_i) <M$$
	for all $i\in\mathbb{N}$.
\end{lemma}
\begin{proof}
	The sequence $\{K_{i}\}$ in \eqref{Kip<0} satisfies
	$$\phi_{\mu_{i}}(\bar{h}_{K_{i}})=\sup \left\{\phi_{\mu_{i}}(f)=\frac{\gamma^{n}([f])}{\int_{\omega_{i}} f^{p} d\mu_{i}}: f \in C^{+}\left(\omega_{i}\right)\right\}.$$
	Let $1|_{\omega_{i}}\in C^{+}\left(\omega_{i}\right)$ denote the unit  function on $\omega_{i}$. For any $i\in\mathbb{N}$, the corresponding Wulff shape satisfies the inclusion
	\begin{align*}
		[1|_{\omega_{i}}]&=C\cap \bigcap_{u \in \omega_{i}}\{x\in\mathbb{R}^{n}: \langle x, u\rangle \le -1\} \\
		&\supset C\cap \bigcap_{u \in \overline{\Omega}}\{x\in\mathbb{R}^{n}: \langle x, u\rangle \le -1\} \\
		&=\bigcap_{u \in \overline{\Omega}}\{x\in\mathbb{R}^{n}: \langle x, u\rangle \le -1\}\\
		&\triangleq A.
	\end{align*}
		Thus $$\gamma^{n}([1|_{\omega_{i}}])\ge\gamma^{n}(A)>0$$
	for all $i\in\mathbb{N}$. Substituting into $\phi_{\mu_{i}}$, since $\mu_{i}\le \mu$, we have
	\begin{align*}
		\phi_{\mu_{i}}(1|_{\omega_{i}})&=\frac{\gamma^{n}([1|_{\omega_{i}}])}{\mu_{i}(\omega_{i})}\\
		&\ge\frac{\gamma^{n}([1|_{\omega_{i}}])}{\mu(\Omega)}\\
		&\ge \frac{\gamma^{n}(A)}{\mu(\Omega)}.
	\end{align*}
	Therefore, $\phi_{\mu_{i}}(\bar{h}_{K_{i}})=\sup \left\{\phi_{\mu_{i}}(f): f \in C^{+}\left(\omega_{i}\right)\right\}\ge \phi_{\mu_{i}}([1|_{\omega_{i}}])\ge  \frac{\gamma^{n}(A)}{\mu(\Omega)}>0$. 
	
	Denote again
	$$r_{i}=\min\{r:rB\cap K_{i}\neq\emptyset \},$$
	Then we also have
	$\bar{h}_{K_{i}}^{p}\ge r_{i}^{p}$.
	 If $r_{i}\rightarrow\infty$,  	 then
	\begin{align*}
		\phi_{\mu_{i}}(\bar{h}_{K_{i}})=&\frac{\gamma^{n}(K_{i})}{\int_{\omega_{i}}\bar{h}_{K_{i}}^{p}  d\mu_{i}}\\
		&\le \frac{\gamma^{n}(\mathbb{R}^{n}-r_{i}B)}{\mu_{i}(\omega_{i}) r_{i}^{p}}\\
		&\le \frac{\gamma^{n}(\mathbb{R}^{n}-r_{i}B)}{\mu_{0}(\omega_{0}) r_{i}^{p}}\\
		&\le\frac{l}{r_{i}}e^{-\frac{r_{i}^{2}}{2n}}\cdot r_{i}^{-p}\rightarrow 0,
	\end{align*}
	for some constant $l$.
	This contradicts the  uniform lower bound estimate  $\phi_{\mu_{i}}(\bar{h}_{K_{i}})\ge  \frac{\gamma^{n}(A)}{\mu(\Omega)}$.   If $r_{i}\rightarrow0$, then
	\begin{align*}
			\phi_{\mu_{i}}(\bar{h}_{K_{i}})=&\frac{\gamma^{n}(K_{i})}{\int_{\omega_{i}}\bar{h}_{K_{i}}^{p}  d\mu_{i}}\\
		&\le \frac{\gamma^{n}(\mathbb{R}^{n}-r_{i}B)}{\mu_{i}(\omega_{i}) r_{i}^{p}}\\
		&\le \frac{\gamma^{n}(\mathbb{R}^{n}-r_{i}B)}{\mu_{0}(\omega_{0}) r_{i}^{p}}\rightarrow 0,
	\end{align*}
this	also arrive at a contradiction.
\end{proof}

\begin{theorem}
	If $p<0$, 
	for any nonzero finite Borel measure  $\mu$ on $\Omega$,  there exists a $C$-pseudo-cone $K$ such that
	$$\mu=-\frac{c}{p}S_{p,\gamma^{n}}(K,\cdot),$$
	where $c=\frac{\int_{\Omega}\bar{h}_{K}^{p}d\mu}{\gamma^{n}(K)}$.
	
\end{theorem}

\begin{proof}
	This follows from the uniform estimate of $\{K_{i}\}$ in Lemma \ref{unip<0}  and applying approximation argument as Theorem \ref{ minkowskip>0} again.
\end{proof}

\section{uniqueness of solution}

The Gaussian measure is log-concave with respect to convex sets (in this paper, we focus on a special class of convex sets---$C$-pseudo-cones), that is:

\begin{lemma}\label{log cancave}
	For  $C$-pseudo-cones $K$ and $L$  in  $\mathbb{R}^{n}$,   $0<t<1$, 
	\begin{equation}\label{logcancave1}
		\gamma^{n}((1-t) K+t L) \geq \gamma^{n}(K)^{1-t} \gamma^{n}(L)^{t}
	\end{equation}
	 equality holds if and only if $K=L$.
\end{lemma}

Similar to the method for obtaining uniqueness in classical convex body theory, the following $L_{p}$ ``mixed volume'' inequality is a necessary technique for establishing the uniqueness of the corresponding $L_{p}$ Minkowski problem. 

\begin{lemma}\label{Minkowski ine}
	Let $\omega\subset\Omega$ be a nonempty compact set  and $p\in (0,1]$, for $K,L\in\mathcal{K}(C,\omega)$, we have
	\begin{equation}
		\frac{1}{p\gamma^{n}(K)}\int_{\omega}\bar{h}_{K}^{p}-\bar{h}_{L}^{p}dS_{p,\gamma^{n}}K\ge \log\frac{\gamma^{n}(L)}{\gamma^{n}(K)},
	\end{equation}
	 equality holds if and only if $K=L$.
\end{lemma}
\begin{proof}
	If $p\in (0,1]$, $f(t)=t^{p}$ is a concave function, consequently, for $t\in[0,1]$,
	$$\left( (1-t)\bar{h}_{K}+t\bar{h}_{L}\right)^{p}\ge (1-t)\bar{h}_{K}^{p}+t\bar{h}_{L}^{p}.$$
	In other words,
	$$ (1-t)\bar{h}_{K}+t\bar{h}_{L}\ge \left( (1-t)\bar{h}_{K}^{p}+t\bar{h}_{L}^{p}\right) ^{\frac{1}{p}},$$
	and
	$$ -\left( (1-t)\bar{h}_{K}+t\bar{h}_{L}\right) \le -\left( (1-t)\bar{h}_{K}^{p}+t\bar{h}_{L}^{p}\right) ^{\frac{1}{p}}.$$
	Therefore, we have the set inclusion
	\begin{align*}
		[\left( (1-t)\bar{h}_{K}^{p}+t\bar{h}_{L}^{p}\right) ^{\frac{1}{p}}\big|_{\omega}]&=C \cap \bigcap_{u \in \omega}\left\{y \in \mathbb{R}^{n}:\langle y, u\rangle \leq -\left( (1-t)\bar{h}_{K}(u)^{p}+t\bar{h}_{L}(u)^{p}\right) ^{\frac{1}{p}}\right\}\\&\supset C \cap \bigcap_{u \in \omega}\left\{y \in \mathbb{R}^{n}:\langle y, u\rangle \leq -\left( (1-t)\bar{h}_{K}(u)+t\bar{h}_{L}(u)\right) \right\}\\
		&=[\bar{h}_{K}|_{\omega}+t(\bar{h}_{L}|_{\omega}-\bar{h}_{K}\big|_{\omega})]\\
		&\supset C \cap \bigcap_{u \in \Omega}\left\{y \in \mathbb{R}^{n}:\langle y, u\rangle \leq -\left( (1-t)\bar{h}_{K}(u)+t\bar{h}_{L}(u)\right)\right\}\\
		&=(1-t)K+tL.
	\end{align*}
	 By the log-concave property in Lemma \ref{log cancave} we get
	
	\begin{align*}
		\log \gamma^{n}\left( \left[\left( (1-t)\bar{h}_{K}^{p}+t\bar{h}_{L}^{p}\right) ^{\frac{1}{p}}\big|_{\omega}\right]\right) 
		&\ge\log \gamma^{n}\left( \left[\bar{h}_{K}|_{\omega}+t(\bar{h}_{L}|_{\omega}-\bar{h}_{K}|_{\omega})\right]\right) \\ &\geq\log \gamma^{n}((1-t) K+t L)\\
		&  \geq (1-t)\log \gamma^{n}(K) +t\log \gamma^{n}(L).
	\end{align*}
	Using the $L_{p}$ variational formula in Lemma \ref{vari formula}, we differentiate both sides of the above inequality at $t = 0^{+}$:
	$$-\frac{1}{p\gamma^{n}(K)}\int_{\omega}\bar{h}_{L}^{p}-\bar{h}_{K}^{p}dS_{p,\gamma^{n}}K\ge \log\gamma^{n}(L)-\log\gamma^{n}(K).$$
Assuming equality holds, it follows that
	\begin{align}\label{num}
		&\lim _{t \rightarrow 0^{+}} \frac{\log \gamma^{n}\left(\left[\left( (1-t)\bar{h}_{K}^{p}+t\bar{h}_{L}^{p}\right) ^{\frac{1}{p}}\big|_{\omega}\right]\right) -\log \gamma^{n}(K) }{t}\\  \nonumber
		=&\lim _{t \rightarrow 0^{+}} \frac{\log \gamma^{n}\left([\bar{h}_{K}|_{\omega}+t(\bar{h}_{L}|_{\omega}-\bar{h}_{K}|_{\omega})]\right) -\log \gamma^{n}(K) }{t}\\  \nonumber
		=&\lim _{t \rightarrow 0^{+}}\frac{\log \gamma^{n}((1-t) K+t L)-\log \gamma^{n}(K)}{t}\\  \nonumber
		=&\log\gamma^{n}(L)-\log\gamma^{n}(K).
	\end{align}
From the log-concave inequality \eqref{logcancave1},	the function $k(t)=\log\left(\gamma^{n}((1-t) K+t L)\right) $ is concave on the interval $[0,1]$. In the case where equality is achieved, $k'(0)=k(1) -k(0)$ in \eqref{num}, then $k(t)$ must be a linear function. As a result,  the equality condition in \eqref{logcancave1} leads to the conclusion that $K=L$.
\end{proof}

	Using the $L_{p}$ mixed volume inequality, we can derive the $L_p$ generalization of the uniqueness result for the Gaussian Minkowski problem for $C$-pseudo-cones in \cite{shan2025}.

\begin{theorem}\label{unique}
For $p\in (0,1]$, 	let $\omega\subset\Omega$ be a nonempty compact set. If	 $K,L\in\mathcal{K}(C,\omega)$ have the same $L_{p}$ Gaussian surface area measure, i.e.,
	$$S_{p,\gamma^{n}}(K,\cdot)=S_{p,\gamma^{n}}(L,\cdot),$$
and if $\gamma^{n}(K)=\gamma^{n}(L)$,	 then $K=L$.
\end{theorem}

\begin{proof}
	 Since $\gamma^{n}(K)=\gamma^{n}(L)$, by Lemma \ref{Minkowski ine}, we get
		$$\int_{\omega}\bar{h}_{K}^{p}-\bar{h}_{L}^{p}dS_{p,\gamma^{n}}K\ge 0.$$
Therefore,
	\begin{equation*}
		\int_{\omega} \bar{h}_{L}^{p} d S_{p,\gamma^{n}}L=\int_{\omega} \bar{h}_{L}^{p} d S_{p,\gamma^{n}}K\le \int_{\omega} \bar{h}_{K}^{p} d S_{p,\gamma^{n}}K.
	\end{equation*}
	By substituting $K$ for $L$ and vice versa,
	\begin{align*}
		\int_{\omega} \bar{h}_{K}^{p} d S_{p,\gamma^{n}}K=\int_{\omega} \bar{h}_{K}^{p} d S_{p,\gamma^{n}}L\le \int_{\omega} \bar{h}_{L}^{p} d S_{p,\gamma^{n}}L.
	\end{align*}
	Thus,
	$$\int_{\omega} \bar{h}_{L}^{p} d S_{p,\gamma^{n}}L\le \int_{\omega} \bar{h}_{K}^{p} d S_{p,\gamma^{n}}K\le\int_{\omega} \bar{h}_{L}^{p} d S_{p,\gamma^{n}}L. $$
	As a result, equality holds in both inequalities. By
	Lemma \ref{Minkowski ine}, we conclude that $K=L$.
\end{proof}

In the absence of a volume constraint, the following lemma extends the non-uniqueness of the Gaussian Minkowski problem to the $L_p$ setting.

\begin{theorem}
If $p<n$,	for any pointed, $n$-dimensional closed convex cone $C$ in $\mathbb{R}^{n}$ and any nonempty compact set $\omega\subset\Omega$, 	there exist distinct sets $K, L \in \mathcal{K}(C, \omega)$ such that their $L_p$ Gaussian surface area measures coincide, i.e., $S_{p,\gamma^{n}}(K,\cdot)=S_{p,\gamma^{n}}(L,\cdot)$, but $K\neq L$.
\end{theorem}
\begin{proof}
For each $v\in\omega$, due to the rotational invariance of the Gaussian measure,  suppose that $v=-e_{n}\in\Omega$. Let $H(t)$ denote the intersection  $C\cap H(e_{n},t)$, and define
	$$A=\{x\in\mathbb{R}^{n-1}:(x,1)\in H(1)\}.$$
For $t>0$, the	Gaussian surface area function $s(t)$ of $H(t)$ (up to a constant) is given by $s(t)=\int_{H(t)}e^{-\frac{|y|^{2}}{2}}d\mathcal{H}^{n-1}(y)$.
	Let $B_{a}^{n-1}$ be a ball in $\mathbb{R}^{n-1}$ with radius $a$, the Gaussian volume estimate of $tB_{a}^{n-1}$, as $t\rightarrow0$, we know
	\begin{equation}\label{n-1 estimate}
	\int_{tB_{a}^{n-1}}e^{-\frac{|y|^{2}}{2}}d\mathcal{H}^{n-1}(y)\sim t^{n-1},\quad\quad t\rightarrow 0.
	\end{equation}
	Since
	\begin{align*}
		\int_{tB_{a}^{n-1}}e^{-\frac{|y|^{2}}{2}}d\mathcal{H}^{n-1}(y)=\int_{S^{n-2}}\int_{0}^{at}e^{-\frac{r^{2}}{2}}r^{n-2}drdu=\mathcal{H}^{n-2}(S^{n-2})\int_{0}^{at}e^{-\frac{r^{2}}{2}}r^{n-2}dr,
	\end{align*}
	and
	$$\frac{\int_{0}^{at}e^{-\frac{r^{2}}{2}}r^{n-2}dr}{t^{n-1}}\rightarrow\frac{a^{n-1}}{n-1}\quad\text{as}\quad t\rightarrow 0.$$
	Then we get the estimate \eqref{n-1 estimate}. There exists a ball $B_{a}^{n-1}$ such that $A\subset B_{a}^{n-1}$, therefore, for $t>0$,
	\begin{align*}
		0<t^{1-p}s(t)&=t^{1-p}\int_{H(t)}e^{-\frac{|y|^{2}}{2}}d\mathcal{H}^{n-1}(y)\\
		&=t^{1-p}\int_{(x,t)\in H(t)}e^{-\frac{|x|^{2}+t^{2}}{2}}d\mathcal{H}^{n-1}(x)\\
		&=t^{1-p}e^{-\frac{t^{2}}{2}}\int_{tA}e^{-\frac{|x|^{2}}{2}}d\mathcal{H}^{n-1}(x)\\
		&\le t^{1-p}e^{-\frac{t^{2}}{2}}\int_{B_{a}^{n-1}}e^{-\frac{|x|^{2}}{2}}d\mathcal{H}^{n-1}(x)\\
		&\sim t^{1-p}\cdot t^{n-1}=t^{n-p}\quad \text{as}\quad t\rightarrow0.
	\end{align*}
	If $p<n$,   we obtain $t^{1-p}s(t)\rightarrow 0$ as $t\rightarrow 0$. If $t \to \infty$, the convergence rate of $s(t)$ is determined by $e^{-\frac{t^2}{2}}$, and thus $t^{1-p}s(t)\rightarrow 0$. Then there exist positive numbers $t_{1},t_{2}>0$, with $t_{1}\neq t_{2}$, such that $t_{1}^{1-p}s(t_{1})=t_{2}^{1-p}s(t_{2})$.
	
	Define $K=H(t_{1})+C$ and $L=H(t_{2})+C$. Both $K$ and $L$ belong to the class $\mathcal{K}(C,v)\subset\mathcal{K}(C,\omega)$.  We then compute $$S_{p,\gamma^{n}}(K,v)=\bar{h}_{K}^{1-p}(v)S_{\gamma^{n}}(K,v)=\frac{1}{(\sqrt{2\pi})^{n}}t_{1}^{1-p}s(t_{1})$$
	and similarly,
	$$S_{p,\gamma^{n}}(L,v)=\bar{h}_{L}^{1-p}(v)S_{\gamma^{n}}(L,v)=\frac{1}{(\sqrt{2\pi})^{n}}t_{2}^{1-p}s(t_{2}).$$
Since	the measures	$S_{p,\gamma^{n}}(K,\cdot)$ and $S_{p,\gamma^{n}}(L,\cdot)$ are both supported on $v$, it follows that	$S_{p,\gamma^{n}}(K,\cdot)= S_{p,\gamma^{n}}(L,\cdot)$.
	 This completes the proof of the desired result.
\end{proof}

\end{document}